\documentclass{amsart}
\usepackage{lineno}
\modulolinenumbers[5]
\usepackage{setspace}
\usepackage{multirow}

\usepackage{hyperref}
\usepackage{amsmath}
\usepackage{amssymb}
\usepackage{amsfonts}
\usepackage{bm}
\usepackage{amsthm}
\usepackage{array,epsfig,fancyhdr}
\usepackage[sectionbib,numbers]{natbib}
\usepackage[utf8]{inputenc}
\usepackage{CJKutf8}
\usepackage{tikz}
\usepackage{mathtools}

\usepackage{xr}
\newtheorem{theorem}{Theorem}[section]
\newtheorem{lemma}[theorem]{Lemma}
\newtheorem{corollary}[theorem]{Corollary}

\theoremstyle{definition}
\newtheorem{definition}[theorem]{Definition}
\newtheorem{example}[theorem]{Example}

\theoremstyle{remark}
\newtheorem{remark}[theorem]{Remark}

\numberwithin{equation}{section}

\DeclareGraphicsExtensions{.pdf,.png}
\usepackage{wrapfig}
\usepackage{lipsum}

\begin{document}

\title[To prove the Four Color Theorem unplugged]{A renewal approach to prove\\ the Four Color Theorem unplugged\\[1.5ex]{\footnotesize Part I:}\\[0.5ex] RGB-tilings on maximal planar graphs}

\author{Shu-Chung Liu}
\address{Institute of Learning Sciences and Technologies, National Tsing Hua University, Hsinchu, Taiwan}
%\curraddr{Institute of Learning Sciences and Technologies, National Tsing Hua University, Hsinchu, Taiwan}
\email{sc.liu@mx.nthu.edu.tw}

%\author{Author Two}
%\address{Mathematical Research Section, School of Mathematical Sciences,
%Australian National University, Canberra ACT 2601, Australia}

%    General info

\date{\today}



\begin{abstract}
This is the first part of three episodes to demonstrate a renewal approach for proving the Four Color Theorem without checking by a computer. The second and the third episodes have subtitles: ``R/G/B Kempe chains in an extremum non-4-colorable MPG'' and ``Diamond routes, canal lines and $\Sigma$-adjustments,'' where R/G/B stand for red, green and blue colors to paint on edges and an MPG stands for a maximal planar graph.  In this first part, we introduce R/G/B-tilings as well as their tri-coexisting version RGB-tiling on an MPG or a semi-MPG. We associate these four kinds of edge-colorings with 4-colorings by 1/2/3/4 on vertices in MPS's or semi-MPG's. Several basic properties for tilings on MPG's and semi-MPG's are developed. Especially the idea of R/G/B-canal lines, as well as canal system, is a cornerstone.  This work started on May 31, 2018 and was first announced by the author~\cite{Liu2020} at the Institute of Mathematics, Academia Sinica, Taipei, Taiwan, on Jan.\ 22, 2020, when the pandemic just occurred. 
\end{abstract}     

\maketitle

\section{Introduction} \label{sec:Introduction}

The famous Four Color Theorem has been challenging and soaking up the mind of many mathematicians for 170 years. The original four-color conjecture was first proposed by Francis Guthrie in 1852, due to his working experience: trying to color the map of counties of England~\cite{MarMou2012}. No doubt a good application was to color the map of the world by assuming no enclave, or made some adjustment to fix very few enclave lands. Graph theory experts soon translated this problem to be a new mathematical model: set every county (or region, country) to be a vertex, and the relation of two neighboring countries to be an edge. A map without enclave lands then turns to a planar graph. In graph-theoretic terms, the Four Color Theorem states that any simple planar graph $G$ has its chromatic number $\chi (G)\leq 4$.  The chromatic number  $\chi$ is defined to be the minimum number of colors assigned individually to every vertex of $G$ such that adjacent vertices have different colors.

The first proof of the Four Color Theorem came out on June 21, 1976. Kenneth Appel and Wolfgang Haken at the University of Illinois~\cite{AppelHaken1977, CharLes2005}:
   \begin{quote}
If the four-color conjecture were false, there would be at least one map with the smallest possible number of regions that requires five colors. The proof showed that such a minimal counterexample cannot exist.  
   \end{quote}
They wrote a program and checked by computer that a minimal counterexample to the four-color conjecture could not exist. In their check, 1,834 possible reducible configurations were achieved their own 4-coloring one-by-one. Since then several different proofs for the Four Color Theorem assisted by computer were contributed to mathematical societies. However, the mathematicians in the field of graph theory never give up the hope to find some kind of artificial proof. A proof for The Four Color Theorem unplugged is still one of the most-wanted academic research in the Math world. 

Also well-known and even widely acclaimed as a textbook proof for the five color theorem was given by Alfred Kempe in 1879~\cite{Kempe1879} and Percy Heawood in 1890~\cite{Heawood1890}. Kempe's proof was once claimed for the Four Color Theorem until 1890, when Percy Heawood indicated an error. In addition to exposing the flaw in Kempe's proof, Heawood proved the five color theorem and generalized the four color conjecture to topological surfaces of arbitrary genus~\cite{Heawood1890}.  Because this proof for five color theorem is so classical and easy to comprehend in graph theory, it is shown in the math world popularly. ``Kempe chains'' given on the title of this article is one of the important tools invented by Kempe in that paper. The title also said ``a renewal approach'', which means we will provide modification on Kempe chains, and offer a new point of view to deal with the (pseudo) extremum as a  non-4-colorable planar graph.  

Another early proposed proof was given by Peter Guthrie Tait in 1880~\cite{Tait1880}.  It was not until 1891  that Tait's proof was shown incorrect by Julius Petersen.  In fact, over his lifetime Tait gave several incomplete or incorrect proofs of the Four Color Theorem. In graph theory, Petersen dedicated two famous contributions: the Petersen graph, exhibited in 1898, served as a counterexample to Tait's proof on the Four Color problem: a bridgeless 3-regular graph is factorable into three 1-factors~\cite{Petersen1898}, and the theorem: a connected 3-regular graph with at most two leaves contains a 1-factor.

The purpose of this study is to offer a renewal method improved from Kempe's classical proof. We try to write the series of three parts self-sustaining and the author only need the basic knowledge on graph theory.

\section{Vertex-colorings and RGB-edge-colorings} \label{sec:RGB1}

Normally we use the natural numbers 1, 2, 3, 4, 5,... to color vertices such that two adjacent vertices have different colors (different number). As a renewal method, the corresponding edge-coloring after a vertex-coloring has been done.

Let us start up with $K_4$ which is a basic graph to color by all $\{1,2,3,4\}$. Originally, the leftmost graph $K_4$ has only vertices and edges; so far no vertex-coloring and no edge-coloring. We provide two different vertex-colorings of $K_4$  with Co[$v_i$:i] and Co[$v_1$:4, $v_2$:3, $v_3$:1, $v_4$:2] which are shown as the middle and the right graphs in Figure~\ref{fig:1234EdgeColor} respectively. For convenience, we will also use vertex-coloring function $f:V\rightarrow \{1,2,\ldots\}$. The corresponding edge-colorings of them are demonstrated in the same time. Basically, red edges (R-edges) are for those edges incident to colors 1 and 3, and those edges incident to colors 2 and 4;  green edges (G-edges) are for 1 and 4, as well as 2 and 3; blue edges (B-edges) are for 1 and 2, as well as 3 and 4. In other words
    \begin{center}
    \begin{tabular}{rcl}
{1-3 and 2-4 edges} & \multirow{4}{3cm}{shall be painted by} 
& red or R, r;\\
{1-4 and 2-3 edges} &  & green or G, g;\\
{1-2 and 4-4 edges} &  & blue or B, b;\\
{uncertain in RGB}&  & black or bl, sometime gray. 
    \end{tabular}
    \end{center}
Among these colored edges, we assign those one incident to 1 by a thick line, because all red edges (as well as green and blue edges respectively) are majorly separated into two sub-classes. For a same edge-color, the edges of the two sub-classes never incident to each other. To assign those edges incident to 1 ``thick'' red, green, blue is not mandatory, because our purpose is just to distinguish the two disconnect sub-classes. The property of disconnecting is the main idea of Kempe's proof. For any graph with a 4-coloring function, we can immediate reach an RGB-edge-coloring. However, the the reverse direction is not correct unless we set up some special circumstance. By the way, we draw the rightmost $K_4$ differently to emphasize that it is planar. Most of time we shall draw planar graphs in this way.
   \begin{figure}[h]
   \begin{center}
%   \begin{tabular}{ c }
%   \hspace*{-12pt}   
   \includegraphics{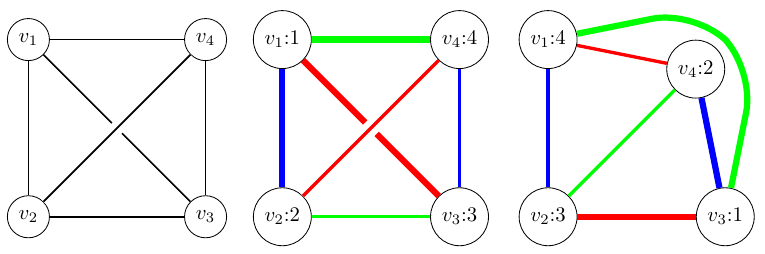}
%   \end{tabular}
   \end{center}
   \caption{The corresponding RGB-edge-colorings} \label{fig:1234EdgeColor}
   \end{figure}

\section{$\mathcal{N}4$, $m\mathcal{N}4$ and $e\mathcal{N}4$} \label{sec:eN4}

From now on all graphs in this article are simple, planar and connected, i.e., no self-loops, no multiple edges and at least a path between any two vertices. Actually multiple edges change nothing for coloring, but they are not only annoying, but also disturb the structure of our planar graphs. Also disconnected graphs cause no difference for the Four Color Map Theorem. Trying to color the map of counties of England and Taiwan in the same time is actually two independent and irrelevant processes.

If the four-color conjecture were false, there would be at least one planar graph $G$ which is non-4-colorable. Let us denote the following three sets under the assumption that  the four-color conjecture were false:
   \begin{eqnarray*}
   \mathcal{N}4 &=& \{G \mid \text{$G$ is a non-4-colorable simple connected planar graph.}\} \\
   m\mathcal{N}4 &=& \{G \mid \text{$G$ is minimal in $N4$ according to the partially ordered}\\
        & &  \text{ set with the inclusion relation of graphs.}\} \\
   e\mathcal{N}4 &=& \{G\in m\mathcal{N}4 \mid \text{$G$ is a minimum w.r.t.\ the number of vertices.}\}    
   \end{eqnarray*}
Clearly, $e\mathcal{N}4 \subseteq m\mathcal{N}4 \subseteq \mathcal{N}4$. Also by the theory of the partially ordered, 
$$e\mathcal{N}4=\emptyset\ \Leftrightarrow\ m\mathcal{N}4 =\emptyset\ \Leftrightarrow\ \mathcal{N}4=\emptyset.$$ 
Of course, the set $e\mathcal{N}4$ turns to be the target of the mathematicians' investigation.

If a planar $G$ has a bridge $ab$ such that $G=G_1 \biguplus \{ab\} \biguplus G_2$, then $G$ is not in $m\mathcal{N}4$, even though $G$ might be non-4-colorable. This idea also works for a planar $G$ with a cut vertex. Therefore, any graph $M\in m\mathcal{N}4$ is at least 2-connected and then any vertex $u\in M$ lays on at least one cycle of $M$. Furthermore, there exist several minimal cycles passing through $u$ and these cycles are called \emph{facets}. In geometry, the name ``facet'' means a triangle, a rectangles, a pentagons, etc., witch is a 2-dimensional substructure belonging to a 3-dimensional polyhedron or polytope. Clearly, the number of facets involving $u$ equals $\deg(u)$ because $M$ is a planar graph.

   \begin{theorem} \label{thm:V5more}
Let $M\in m\mathcal{N}4$. (a) Any supergraph of $M$ is non-4-colorable and any non-trivial subgraph of $M$ is 4-colorable. (b) Any vertex $v\in M$ must have $\deg(v)\ge 5$.   
   \end{theorem}           
   \begin{proof}
(a): About the non-trivial subgraph, it is the basic minimal property in a partially ordered set. As for  
any supergraph of $M$, say $M'$, we can first color $M$. Only at this stage, 4 different colors are not enough. So $\chi(M')>4$.        

(b): By the definition of $m\mathcal{N}4$, we know $M$ is non-4-colorable and $M-\{v\}$ is 4-colorable for any vertex $v\in m\mathcal{N}4$. For 2-connected, no vertices in $M$ are degree $1$. Suppose there is a vertex $v\in m\mathcal{N}4$ with $\deg(v)=2$ or $\deg(v)=3$ then the 4-colorability of $M-\{v\}$ will make $M$ 4-colorable.

Suppose there is a vertex $v\in M$ with $\deg(v)=4$. Let $v_1, v_2, v_3, v_4$ are four neighbors of $v$ clockwise around. Due to the fact that $M$ is non-4-colorable and $M-\{v\}$ is 4-colorable, we shall have a 4-coloring function $f:V(M-\{v\})\rightarrow \{1,2,3,4\}$ such that $f(v_i)=i$ for $i=1,2,3,4$ and then $f(v)=5$ is inevitable. In $M-\{v\}$, let us finish the corresponding RGB-edge-coloring. Between $v_1$ and $v_3$ or between $v_2$ and $v_4$, there is ``at less one'' pair which is red-edge-disconnected, because a pair being red-edge-connected (or simply red-connected) will block another pair's connection. Without loss of generality, suppose  $v_2$ and $v_4$ are red-disconnected. Now by switch colors 2 and 4 over the red-connected component containing $v_4$, we get a new a 4-coloring function $f':V(M-\{v\})\rightarrow \{1,2,3,4\}$ that keeps the same colors on $v_1, v_2, v_3$ as $f$, but $f'(v_4)=2$. Finally by assigning $f'(v)=4$, we reach $f'$ a 4-coloring of $M$ and a contradiction. Now we guarantee that $\deg(v)\ge 5$ for every $v\in M$.   
   \end{proof}
   \begin{remark}
Later we will show that ``at less one'' in this proof is actually ``exactly one'' if we deal with triangulated planar graphs which is the next topic.  
   \end{remark}

\section{Maximal planar graphs and $e\mathcal{MPGN}4$}
\label{sec:eMPGN4}

We introduce triangulated graphs or MPG's, the abbreviation standing for maximal planar graphs, by offering two equivalent definitions:
   \begin{itemize}
   \item A planar graph $G$ is said to be triangulated (also called maximal planar) if the addition of any edge between two non-adjacent vertices of $G$ results in a non-planar graph.
   \item (A much simple definition)  As a planar graph of $G$, all facets of $G$ are triangles, including the outer one.
   \end{itemize}
   
Then the following property is a good reason for us to narrow down our target graphs.   
   
   \begin{theorem}     \label{thm:4colorMPG}
All planar graphs are 4-colorable if and only if all MPG's are 4-colorable.   
   \end{theorem}
   \begin{proof}\
[$\Rightarrow$]: This direction is trivial.\newline
[$\Leftarrow$]: Given a planar graph $G$, we add some edges to make it an MPG, say $H$.  As an MPG, $H$ is 4-colorable, then by removing these “some edges” we turn back to G and it adopts the 4-coloring of $H$.
   \end{proof}

The purpose is to prove the Four Color Theorem; but assuming $e\mathcal{N}4$, $m\mathcal{N}4$ and $\mathcal{N}4$ non-empty (especially the first one) is a regular way of proving by contradiction. Let us denote a new set:
   \begin{eqnarray*}
   e\mathcal{MPGN}4 &=& \{E \in\mathcal{N}4 \mid \text{ $E$ is an MPG and a minimum w.r.t.}\\ & & \text{ the number of vertices.}\}    
   \end{eqnarray*}
\noindent
Compared with $e\mathcal{N}4$, the new set $e\mathcal{MPGN}4$ is better to deal with. The fundamental relation between  $e\mathcal{MPGN}$ and $e\mathcal{N}4$ is a well-known result. For self-sustaining of this paper, we will redo the proof in the rest of this section.

   \begin{quote}
To prove the Four Color Theorem by contrapositive, in the whole paper and our study in future, we always assume $e\mathcal{MPGN}4$ non-empty and discuss many different kinds of situations with much more details set up for an $EP\in e\mathcal{MPGN}4$. For instance, in Part III of this paper, we exam the situation that two degree 5 vertices are neighbors in $EP$.  
   \end{quote}

A well-known fact: any simple planar graph can be triangulated and still simple. It can be proved by induction. Given any planar graph $G$, let $\hat{G}_i$ be one of triangulated planar graph of $G$ by linking edges for some non-adjacent pairs of vertices along $n$-gons ($n>3$). Let  $\hat{e\mathcal{N}4}:= \{ \hat{G}_i \mid G\in e\mathcal{N}4, i=1,2,\ldots \}$. Because $e\mathcal{N}4$ consists of minimum graphs  $G\in m\mathcal{N}4$ w.r.t.\ the number of vertices, the following lemma is trivial. 

   \begin{lemma} \label{thm:eN4}
We have $\hat{e\mathcal{N}4} \subseteq e\mathcal{N}4$, and $\hat{e\mathcal{N}4}$ is non-empty if and only if $e\mathcal{N}4$ is non-empty.  
   \end{lemma}

   \begin{theorem} \label{thm:eMPG4}
{\rm (a)} $e\mathcal{MPGN}4 = \hat{e\mathcal{N}4} = e\mathcal{N}4$.\ \ {\rm (b)} If  $EP\in e\mathcal{MPGN}4$ then any non-trivial subgraph of $EP$ is 4-colorable. Also any MPG graph $G$ with $|G|<|EP|$  is 4-colorable.
   \end{theorem}
   \begin{proof}   
If $e\mathcal{N}4$ is empty then all three sets are empty; therefore (a) holds and (b) has nothing to check. 

Suppose $e\mathcal{N}4$ non-empty. Let $G\in e\mathcal{N}4$ and then we choose one $\hat{G}\in \hat{e\mathcal{N}4}$. 
Notice that $e\mathcal{MPGN}4$ is now non-empty because of $\hat{G}$. Let $EP\in e\mathcal{MPGN}4$. By definition, we relation on the order (the number of vertices) of these three graphs as $|G|=|\hat{G}|\ge |EP|$.
  
On the other hand, we know the fact $e\mathcal{MPGN}4 \subseteq \mathcal{N}4$, and let us consider the minimum orders of both sides. So we have $|EP|\ge |G'|$ for any $G'\in e\mathcal{N}4$. Finally we obtain $|EP|= |G|$ for any $G\in e\mathcal{N}4$, and  we conclude that  $e\mathcal{MPGN}4 =\hat{e\mathcal{N}4}  \subseteq e\mathcal{N}4$. 

Furthermore, any $EP\in e\mathcal{MPGN}4  \subseteq e\mathcal{N}4 \subseteq m\mathcal{N}4$ obeys a property of $m\mathcal{N}4$: any non-trivial subgraph of $EP$, particularly removing some edges from $EP$, must be 4-colorable. Therefore, we conclude (b) and  $e\mathcal{MPGN}4 = e\mathcal{N}4$.   
   \end{proof}  
   
   \begin{corollary} \label{thm:V5more2}
Let $EP\in e\mathcal{MPGN}4$. (a) Any supergraph of $EP$ is non-4-colorable and any non-trivial subgraph of $EP$ is 4-colorable. (b) Any vertex $v\in EP$ must have $\deg(v)\ge 5$.   
   \end{corollary}

We would like to show another poof for the part (b) of Theorem~\ref{thm:eMPG4}. We need a trivial lemma as follows about non-trivial 3-cycles. Also, we will offer another interesting theorem about non-trivial 4-cycles in Part II of this paper.

   \begin{lemma}   \label{thm:nontrivial3}
There is no non-trivial 3-cycle in $EP$, i.e., every 3-cycle in $EP$ forms a 3-facet.  
   \end{lemma}
   \begin{proof}
Suppose there is a non-trivial 3-cycle $\Omega:=a$-$b$-$c$-$a$ in $EP$. By $\Omega$, our $EP$ is separated into two non-empty regions: $\Sigma$ (inside) and $\Sigma'$ (outside) with  $\Sigma\cap\Sigma'=\Omega$. Both $\Sigma$ and $\Sigma'$ are still two MPG's with $|\Sigma|, |\Sigma'|\ge 4$. Also $|\Sigma|$ and $|\Sigma'|$ are less than $|EP|$, thus $\Sigma$ and $\Sigma'$ are 4-colorable with coloring map $f$ and $f'$. Because the intersection $\Sigma\cap\Sigma'$ consists of only three vertices, we can  manipulate vertex-color-switching on $f'$ to get $f(x)=f'(x)$ for $x=a,b,c$. Now $f$ and $f'$ together offer a 4-coloring map of $EP$, and this is a contradiction. 
   \end{proof}
\noindent

   \begin{proof} (Another proof for Theorem~\ref{thm:eMPG4}(b)) A maximal non-trivial subgraph of $EP$ is obtained by removing any edge, say $uv$ in $EP$. Once we show every maximal non-trivial subgraph of $EP$ is 4-colorable, then all other subgraphs of $EP$ are 4-colorable. Let $\deg(u)=k\ge 5$ (by Theorem~\ref{thm:V5more}) with the neighbors $v=v_1, v_2,\ldots, v_k$ clockwise around $u$. Now we form a new MPG, say $G$, from $EP$ by removing vertex $u$ and connecting new edges $vv_3, vv_4,\ldots, vv_{k-1}$. Notice that the edge $vv_i$ for $i=3,4\ldots,k-1$ is valid if $v$ and $v_i$ are not adjacent in $EP-\{u\}$. What happens if $v$ and $v_i$ are adjacent? It is impossible; otherwise we would see a non-trivial 3-cycle $v$-$v_i$-$u$-$v$ in $EP$, and this violates Lemma~\ref{thm:nontrivial3}. 
   
With these new edges, we obtain a new MPG, say $G$, whose order is less than $EP$. So $G$ has a 4-coloring map $f$ due to the definition of $e\mathcal{MPGN}4$. For the maximal nontrivial $EP-\{uv\}$, we need to keep same coloring for the vertices in $G$  and assign $f(u)=f(v)$. This $f$ is definitely a 4-colorable map for $EP-\{uv\}$.  
    \end{proof}
    \begin{remark}[Very important]
Before link a new edge between  $v$ and $v_i$, we must make sure that $v$ and $v_i$ are not adjacent in the current graph. 
    \end{remark}

\section{Tilings and 4-colorings}    \label{sec:RGB2}

A \emph{semi-MPG} is nearly an MPG but it has some facets, called \emph{outer facets}, which form the border of this  structure as a planar graph. Usually an outer facet is an $n$ sided polygon ($n$-gon) with $n\ge 4$. However, in some rare cases we allow $3$-gons to be outer facets if we really need them to play as the border of this structure. (See Sections~\ref{sec:RGB4coloring} and~\ref{sec:Grandline}.)

If there is a single outer facet of $n$ sided polygon, 
   \begin{figure}[h]
   \begin{center}
   \includegraphics[scale=1]{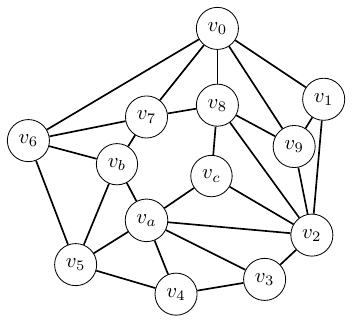}
   \end{center}
\caption{a $(5,7)$-semi-MPG}   
   \end{figure}
then we call it an $n$-semi-MPG. Generally an $(n_1,n_2,\ldots,n_k)$-semi-MPG has $k$ outer facets with sizes $n_1,n_2,\ldots,n_k$ respectively. Most of the time, we forbid any two outer facets sharing a same edge, because we need every edge must belong to one or two triangles. We do have exceptions in Lemma~\ref{thm:evenoddRGB}(a') and in some other place.

Particularly we call an MPG or an $n$-semi-MPG \emph{One Piece}, because all loops in these structures are free loops who can be shrunk into one point in topology. Of course, an $(n_1,n_2,\ldots,n_k)$-semi-MPG is not One Piece, unless we allow two outer facets sharing a same edge.  We draw a brief sketch:
   \begin{center}
   \begin{tabular}{rcl}
  & MPG & \multirow{2}{2.4cm}{{\huge \} }\quad One Piece}\\
\multirow{2}{2.5cm}{semi-MPG\quad {\huge \{ }} & $n$-semi-MPG\\  
    & $(n_1, n_2, \ldots n_k)$-semi-MPG   
   \end{tabular}
   \end{center}

The basic elements in an MPG or a semi-MPG are triangles. Two triangles sharing an edge is call a \emph{diamond}. In the following, the edge shared by two triangles is colored by red. We call it a \emph{red tile} or a \emph{red diamond}. We also call the left graph in Figure~\ref{fig:reddiamond} $v_2v_4$-diamond and the four surrounding edges of $v_2v_4$ are $v_1v_2$, $v_2v_3$, $v_3v_4$ and $v_4v_1$. If a triangle with only one red edge, then we call it a \emph{red half-tile} or a \emph{red triangle}. Of course, there are also green or blue tiles and half-tiles. For a fixed edge-color, there is exactly one edge of such color for every single triangle.  
\noindent
   \begin{figure}[h]
   \begin{center}
   \includegraphics{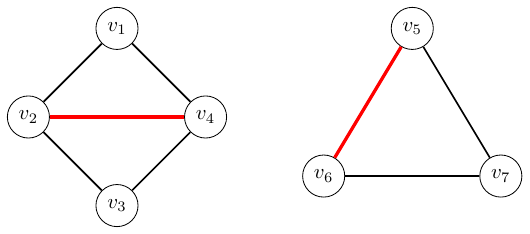}
   \end{center}
   \caption{A red tile $v_2v_4$-diamond and a red half-tile} \label{fig:reddiamond}
   \end{figure}

Let $G$ be an MPG or a semi-MPG. We try to edge-color exactly one of three edges of every triangle by red. We call the process and a possible result an \emph{R-tiling} on $G$. The Figure~\ref{fig:twoRtilings} in the following  
\noindent
   \begin{figure}[h]
   \begin{center}
   \begin{tabular}{ c c }
   \includegraphics[scale=0.9]{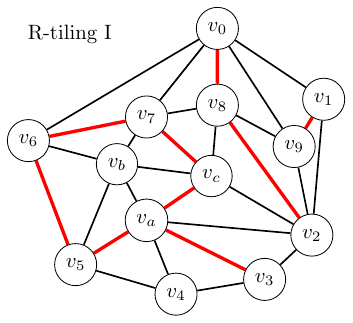}\
   \includegraphics[scale=0.9]{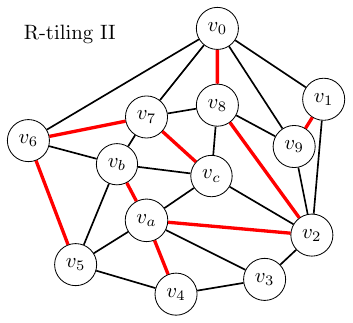}
   \end{tabular}
   \end{center}
   \caption{Two R-tilings; the right one induces a 4-coloring.}  \label{fig:twoRtilings}
   \end{figure}
shows R-tiling I and R-tiling II on a same graph $G$. 
   \begin{definition} \label{def:Rtiling}
An R-tiling is actually a function $T_r:E(G)\rightarrow \{\text{red, black}\}$, where $E(G)$ is the edge set of $G$, such that every triangle of $G$ has one one red and two black edges. 
   \end{definition}
Now we have the formal reason why we forbid any two facets of $G$ sharing a same edge, because to make tiling every edge must belong to one or two triangles. 

For short, we use r or R to stand for red, g or G for green, b or B for blue, and bl for black. Black color also means uncertain for R/G/B. Not just R-tilings, we work with G-tilings and B-tilings. If we only consider a single-color tiling, then R, G and B are independent. However, most of time we wish them would coexist. That means each triangle in an MPG or a semi-MPG shall have three different edge-colors: R, G and B. If three different color-tilings coexist, then we call this representation an \emph{RGB-tiling}, denoted  by $T_{rgb}:E(Q)\rightarrow \{\text{red, green, blue}\}$ and then no more black color for any edge.      
   
Notice that there is a red 5-cycle, namely $v_5$-$v_6$-$v_7$-$v_c$-$v_a$-$v_5$, in R-tiling I. That means no coexisting G-tiling as well as B-tiling inside this 5-cycle. Also notice that there no red odd-cycle in R-tiling II; thus a 4-coloring can be induced by R-tiling II. The rightmost below is an RGB-tiling which is an example to fulfill coexisting G-tiling and B-tiling w.r.t.\ R-tiling II. As for R-tiling I, it tells another story. Even though no fully 4-coloring of $G$ by using R-tiling I, 
\noindent
   \begin{figure}[h]
   \begin{center}
   \begin{tabular}{ c c }
   \includegraphics[scale=0.9]{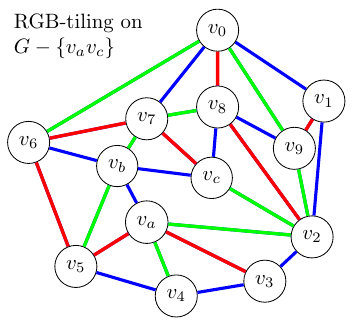}\
   \includegraphics[scale=0.9]{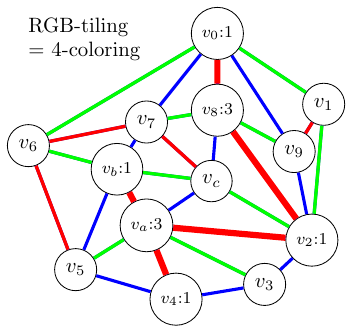}
   \end{tabular}
   \end{center}
   \caption{Two RGB-tilings; the right one induces a 4-coloring.} \label{fig:twoRGBtiling}
   \end{figure}
we may sacrifice some edges to reach a minor result. The leftmost graph is what we just said a minor result. It is an RGB-tiling on $G-\{v_av_c\}$. Because $v_av_c$ lays on that red 5-cycle, removing it makes the new $(4,7)$-semi-MPG 4-colorable generated by R-tiling I. This story has an interesting implication: Of course, if we return $v_av_c$ as a red edge, then the red 5-cycle is back. Amazingly, if we return $v_av_c$ as a green edge, then we create a green odd-cycle.  If we return $v_av_c$ as a blue edge, then we create a blue odd-cycle. This fact is very important.

Notice that tilings make sense only in an MPG or a semi-MPG. In a normal 4-colorable graph, we introduce RGB-edge-colorings in Section~\ref{sec:RGB1}. The concept of RGB-tilings is advanced compared with the one of RGB-edge-colorings when we deal with (nearly) triangulated graphs. So far we have not guaranteed that an MPG or a semi-MPG does have an R-tiling. We just showed some examples for R-tilings and induced possible RGB-tilings. The existence of R-tilings is a topic in the next section when we focus on $e\mathcal{MPGN}4$.

\section{RGB-tilings and 4-coloring on MPG's or $n$-semi-MPG's}    \label{sec:RGB4coloring}

Let $M$ be an $(n_1,n_2,\ldots,n_k)$-semi-MPG. It has $k$  \emph{outer facets} with sizes $n_i$, i.e., each outer facet as a piece of the border of this planar graph $M$ is an $n_i$ sided polygon ($n_i$-gon). Normally we have $n_i\ge 4$ because there are triangles nearly everywhere in $M$; however, \underline{3-gon outer facets are allowed} in this and the next sections if we do point them out precisely. As a 3-gon outer facet, it does not need to follow Definition~\ref{def:Rtiling} for any provided R-tiling $T_r$. Also notice that Lemma~\ref{thm:nontrivial3} guarantees no non-trivial 3-cycle in $EP$.  After Lemma~\ref{thm:evenoddRGB}, we allow \underline{two outer facets to share edges} in $M$ in this and the next sections. As for edge counting, those shared edges shall be count by multiplicity 2.

   \begin{remark}
In this paper, a non-trivial triangle in an MPG or in a semi-MPG is either a $3$-gon outer facet or a triangle with   vertices and edges both inside and outside. A trivial 4-cycle is defined that its four vertices are the surrounding four of a diamond. 
   \end{remark}

Even though we have not yet prove that an R-tiling must layout in $M$, we can show some properties under the assumption that an R-tiling exists. 
   \begin{lemma}  \label{thm:evenoddRGB}
Let $M$ be an $n$-/$(n_1,n_2,\ldots,n_k)$-semi-MPG for $n, n_i\ge 3$.
  \begin{itemize}
\item[(a)] If $M$ has an R-tiling, then the number of black edges along $\Omega(M)$ must be even. Of course, the rest edges along $\Omega(M)$ are red and each associates with a red half-tile.

\item[(a')] This is a supplemental item w.r.t.\ {\rm (a)} and here we allow two outer facets to share edges in $M$. The index $\sum_{i=1}^k n_i$ counts such shared edges by multiplicity 2. Also a single edge that is shared by two outer facets associates with no triangles; so their colors are free to choose from red and black. Now the result in item {\rm (a)} still works here.

\item[(b)] If $M$ has an RGB-tiling, then the three groups of edges along $\Omega(M)$ sorted by red, green and blue must be either all even or all odd in cardinality. Particularly, when $|\Omega(M)|$ is even then
they are all even, and when $|\Omega(M)|$ is odd then they are all odd.   
   \end{itemize}
   \end{lemma}
   \begin{proof}
(a): First we shall deal with those red edges along $\Omega(M)$. Such an edge $e$ is a red half-tile. Let us add up to two extra adjacent black edges (with ``V'' shape) to complete a red $e$-diamond.  Now we see a new semi-MPG, say $\hat{M}$, and a new R-tiling that is made by red diamonds, where all edges along $\Omega(\hat{M})$ are black. A simple equation 
   $$|\Omega(\hat{M})| =4\#(\text{red diamonds})-2\#(\text{black edges inside $\hat{M}$})$$
verifies that $|\Omega(\hat{M})|$ is even. Now removing all ``V'' shapes from $\Omega(\hat{M})$, we prove that the number of black edges along $\Omega(M)$ must be even.

(a'): For this supplemental item, we simply remove all edges that are shared by two outer facet and get a new semi-MPG $\bar{M}$. Apply (a) on $\bar{M}$ and then the number of black edges along $\Omega(\bar{M})$ must be even. How about the removed edges originally along the outer facets of $M$? Because the counting index $\sum_{i=1}^k n_i$ counts those edges shared by two outer facets with multiplicity 2, and so do the counting for those black edges shared by two outer facets of $M$. The proof of this supplemental item is complete.

(b): The concept is easy. In an RGB-tiling which made by coexisting R-tiling, G-tiling and B-tiling. In view of R-tiling, green and blue edges are treated as black. So, this property is a corollary of (a).     
   \end{proof}

   \begin{remark}  \label{re:SupplementalItem}
So far we have no supplemental item (b') for Theorem~\ref{thm:evenoddRGB}. Those edges shared by two outer facets have no rule to color R, G and B in order to reach an ``good'' RGB-tiling. We could modify the definition of an RGB-tiling by additionally requiring that the numbers of red, green and blue must be either all even or all odd along \underline{each} outer facet. With this new definition, removing some shared edges to modify outer facets or changing colors on some shared edges will still keep all-even/all-odd property; because any shared edge is counted by multiplicity 2. Wait! do we really need this stronger definition? Or actually all even or all odd property on $\Omega(M)$ is always true like Lemma~\ref{thm:evenoddRGB}(a'). For this ``stronger definition'', please refer to Theorem~\ref{thm:4RGBtilingGeneralization}; as for a new supplemental item (b'), please see the comming  Corollary~\ref{thm:evenoddRGB2}. Also Remark~\ref{re:CSigma} offers a more advance point of view.  
   \end{remark}

   \begin{corollary}  \label{thm:evenoddRGB2}
Let $M$ be an $n$-/$(n_1,n_2,\ldots,n_k)$-semi-MPG for $n, n_i\ge 3$ with an RGB-tiling $T_{rgb}$. Here we allow two outer facets to share edges in $M$. Along $\Omega(M)$, the three numbers of red, green blue edges are all even if $|\Omega(M)|$ is even; all odd if $|\Omega(M)|$ is odd. Notice that any shared edge is counted by multiplicity 2. 
   \end{corollary}
   \begin{proof}
We simply remove those shared edges and obtain a new semi-MPG without any shared edge. Then apply Lemma~\ref{thm:evenoddRGB}(b). It does not matter what colors are on those shared edges, because they shall be counted by multiplicity 2.   
   \end{proof}
   
   \begin{remark} \label{re:CSigma}
Let $M$ be an MPG or an $n$-/$(n_1,n_2,\ldots,n_k)$-semi-MPG for $n, n_i\ge 3$.  Here we allow two outer facets to share edges in $M$. Let $C$ be any cycle in $M$ that might pass through some edges along $\Omega(M)$ except those shared edges. Let $\Sigma^+$ ($\Sigma^-$) denote the subgraph of $M$ inside (outside) of $C$. Both $\Sigma^+$ and $\Sigma^-$ are semi-MPG's with an outer facet $C$. Once we have a $T_r$ or $T_{rgb}$ on $M$, then we can apply Lemma~\ref{thm:evenoddRGB} and Corollary~\ref{thm:evenoddRGB2} on $\Sigma^+$ and $\Sigma^-$.  
   \end{remark}

   \begin{example} \label{ex:Counterexample1}
There are two red-connected components on the following graph and these red edges form an R-tiling. However, is this R-tiling good? What criteria do we think about ``good''? Of course, a good R-tiling means easy to paint by 4 colors or easy to tell that we have to use the color 5. So far we just want to experience some examples. 
   \begin{figure}[h]
   \begin{center}
   \includegraphics[scale=0.9]{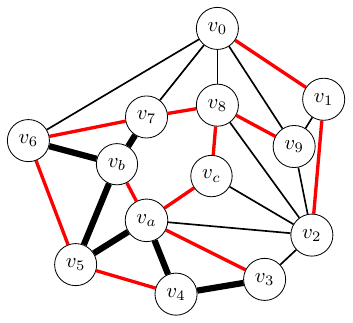}
   \end{center}
\caption{Along 5 or 7 facet, the numbers of black edges are both odd.}  \label{fig:57notGrand}
   \end{figure}
In this graph, we see several thick black edges linking two vertices in a single red-connected component; so there are many different cycles that have nearly all red edges but only one thick black edge. There are also some cycles who have more odd number of black edges; for instance, the cycle $C: v_4$-$v_5$-$v_b$-$v_7$-$v_8$-$v_c$-$v_a$-$v_4$ is one of them. We also see that there are a 5-gon and a 7-gon outer facets; the former has one black edge and the latter has three black edges. By Remark~\ref{re:CSigma}, any cycle $C'$ that separates the 5-gon and the 7-gon on the two side of itself will have odd number of black edges along $C'$. Any cycle $\bar{C}$ who makes the 5-gon and the 7-gon on the same side of itself will have even number of black edges along $\bar{C}$.    
   \end{example}

From Lemma~\ref{thm:evenoddRGB}, Corollary~\ref{thm:evenoddRGB2} and Example~\ref{ex:Counterexample1}, we find that an MPG and an $n$-semi-MPG are easy to deal with. That is why we call them \emph{One Piece} on purpose. 

   \begin{theorem}[The First Fundamental Theorem v1: for R-/RGB-tilings and 4-colorability] \label{thm:4RGBtiling}
Let $M$ be an MPG or an $n$-semi-MPG ($n\ge 4$). Then the following are equivalent:
   \begin{itemize}
\item[(a)] $M$ is 4-colorable.
\item[(b)] $M$ has an RGB-tiling.
\item[(c)] $M$ has an R-tiling without red odd-cycle.
   \end{itemize}
   \end{theorem}
   \begin{proof}
The diagram of the whole proof is (a) $\Rightarrow$ (b) $\Rightarrow$ (c) $\Rightarrow$ (a). However, we do not have enough tools to prove [(c) $\Rightarrow$ (a)] clearly.  So, we just leave a rain check for a while and this part will be proved in the next section right after Theorem~\ref{thm:RtilingOnePiece}.  

[(a) $\Rightarrow$ (b)]: A 4-coloring function on any graph induces an RGB-edge-coloring. Since $M$ is an MPG or an $n$-semi-PG, an RGB-edge-coloring of $M$ is actually an RGB-tiling on $M$.

[(b) $\Rightarrow$ (c)]: Suppose $C$ is any red cycle in $M$. Since $M$ is an MPG or an $n$-semi-MPG, we see at least one of the two sides of $C$ forms a $|C|$-semi-MPG with outer facet $C$. The red is the only color along this outer facet; thus by Lemma~\ref{thm:evenoddRGB}(b) (or Corollary~\ref{thm:evenoddRGB2}) the length $|C|$ must be even for no green and blue along $C$.
    \end{proof}
    
   \begin{corollary} \label{thm:4RGBtiling2}
Let $M$ be an MPG or an $n$-semi-MPG ($n\ge 4$). The graph $M$ is non-4-colorable if either there is no R-tiling on $M$ or every R-tiling on $M$ has at least a red odd-cycle.
   \end{corollary}

Theorem~\ref{thm:4RGBtilingGeneralization} in the the next section is a generalization of this theorem. Although the proof of [(c) $\Rightarrow$ (a)] is still a rain check, we apply it as the following corollary.    

   \begin{corollary} \label{thm:atoc}
In an MPG or an $n$-semi-MPG ($n\ge 4$), if we have any two coexisting R/G/B-tilings, then the third coexisting single coloring tiling is immediately ready. Also there are no red/green/blue odd-cycles.
   \end{corollary}   
    
Since $e\mathcal{MPGN}4 = e\mathcal{N}4$ and by definition, all MPG's in this set have same number of vertices (same order). Let $\omega$ be this order. For convenience we usually use $EP$ to denote one of the extremum planar graph in $e\mathcal{MPGN}4$. The order of $EP$ is $\omega$. If $G$ is a planar graph (not necessarily an MPG or a semi-MPG) such that $|G|<\omega$ or $G$ is a subgraph of $EP$ with less number of edges, then $G$ must be 4-colorable. And then this 4-coloring function of $G$ induces an R-tiling on $G$ without red odd-cycles.   

   \begin{remark}
In set theory, ordinal numbers and cardinal numbers are so fundamental and fancy in mathematics. We use $|EP|=\omega$ in purpose, because $\omega$ is the first ordinal number corresponding to infinity. 
   \end{remark}

   \begin{example}[Counterexample] \label{ex:Counterexample2}
Here a counterexample is provided to show that ``(c) $\Rightarrow$ (a)'' of Theorem~\ref{thm:4RGBtiling}  might not work for an $(n_1, n_2)$-semi-MPG.
   \begin{figure}[h]
   \begin{center}
   \includegraphics[scale=0.9]{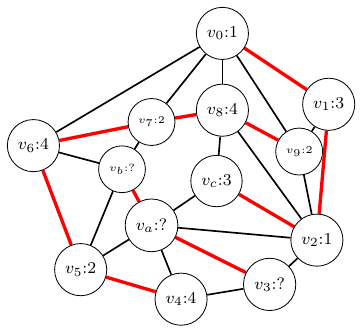}
   \end{center} 
\caption{Odd number of black edges along 5 and 7 facets}  
\label{fig:counterEX2}
   \end{figure} 
We show a $(5,7)$-semi-MPG and an R-tiling without red odd-cycle on it in Figure~\ref{fig:counterEX2}. According to the rule, red edges are used to connect vertices with colors 1 and 3, or colors 2 and 4. Without loss of generality, a 4-coloring on the red-connected components $rC_{v_0}:=\{v_0, v_1, v_2, v_c\}$ and $rC_{v_8}:=\{v_4, v_5,\ldots, v_9\}$ has been done. Now we find that there is not proper way to color vertices $v_3, v_a, v_b$, particularly for $v_a$; unless we are allowed to set coloring function $f(v_a)=5$. This example shows that on $(n_1, n_2,\ldots, n_k)$-semi-MPG providing an R-tiling without red odd-cycle is not enough to achieve a 4-coloring function. The example also shows that three (or odd number of) red-connected components $rC_{v_0}$, $rC_{v_8}$ and $rC_{v_3}:\{v_3, v_a, v_b\}$ ``in a loop'' either along the 5-gon or the 7-gon will cause problems. 
   \end{example}
   
   \begin{example}[Counterexample] \label{ex:Counterexample3}
Here we give another counterexample to show that ``(b) or (c) $\Rightarrow$ (a)'' of Theorem~\ref{thm:4RGBtiling} might not work for an $(n_1, n_2)$-semi-MPG. We show a $(5,5)$-semi-MPG and assign a fixed R-tiling $T_r$ without red odd-cycle as the left graph in Figure~\ref{fig:counterEX3}. Let us try to color green or blue for the rest black edges.  Even if we can extend $T_r$ to all kinds of different coexisting $T_{rgb}$, but none of these $T_{rgb}$ can induce a 4-coloring function. 
   \begin{figure}[h]
   \begin{center}
   \begin{tabular}{ccc}
   \includegraphics[scale=0.76]{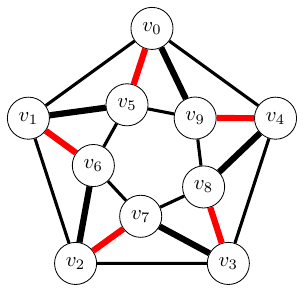}&
   \includegraphics[scale=0.76]{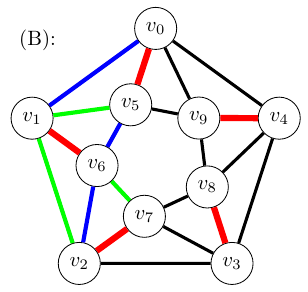}&
   \includegraphics[scale=0.76]{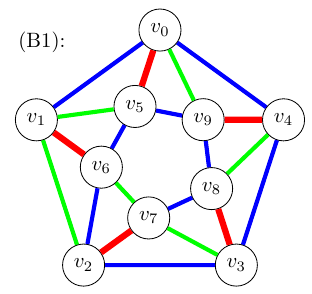}
      \end{tabular}
   \end{center}
\caption{Given $T_r$, let us finish it by (A) and (B). Here is (B1).}  
\label{fig:counterEX3}      
   \end{figure}
In particular, we focus on the element in $X:=\{v_1v_5, v_2v_6, v_3v_7, v_4v_8, v_0v_9\}$. Without loss of generality, we assume that among $X$ the number of green edges is greater than the number of blue edges. (A) If all five edges in $X$ are green then $v_0$-$v_1$-$v_2$-$v_3$-$v_4$-$v_0$ and $v_5$-$v_6$-$v_7$-$v_8$-$v_9$-$v_5$ must be two blue 5-cycles; thus we cannot finish (A) with a 4-coloring function from this RGB-tiling.  (B) Suppose edges in $X$ use both green and blue. Without loss of generality, we set $v_1v_5$ green and $v_2v_6$ blue, and then $v_0v_1$ and $v_5v_6$ must be blue, and $v_1v_2$ and $v_6v_7$ must be green. This temporary status is shown as the middle graph in Figure~\ref{fig:counterEX3}. According to the possible colors of the rest three edges in $X$, we need only consider the following three subcases: 
   \begin{figure}[h]
   \begin{center}
   \begin{tabular}{cc}
   \includegraphics[scale=0.8]{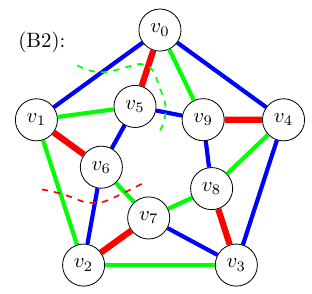}\quad &
\quad  \includegraphics[scale=0.8]{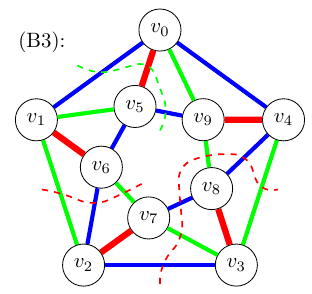}
      \end{tabular}
   \end{center}
\caption{(B2) and (B3)}  
\label{fig:counterEX31}  
   \end{figure}
(B1) Let all three of $v_3v_7$, $v_4v_8$, $v_0v_9$ be green. Even though this is an RGB-tiling, the blue path from $v_1$ to $v_7$ cannot induce a 4-coloring. It is quite interesting that a single blue-connected component ``in a loop'' either along the 5-gon or the 7-gon will cause problems. (B2) Let $v_3v_7$ be blue and $v_4v_8$, $v_0v_9$ green and we get the left graph in Figure~\ref{fig:counterEX31}. The bug occurs because three green-connected components ``in a loop''; so, there is no 4-coloring. (B3) Let $v_4v_8$ be blue and $v_3v_7$, $v_0v_9$ green, and then we get the right graph in Figure~\ref{fig:counterEX31}. Again, the bug occurs because three green-connected components ``in a loop''; so, there is no 4-coloring.  Finally, we make a conclusion: Once we assign a fixed R-tiling $T_r$ without red odd-cycle as the left graph in Figure~\ref{fig:counterEX3}, there are already five red-connected components ``in a loop'' along  both two 5-gons. By the way, the red and green dashed-lines in (B2) and (B3) are drawn for Example~\ref{ex:Conquer} in advance, not now.
   \end{example}

\section{Grand canal lines over One Piece}    \label{sec:Grandline}

In this and the last sections, \underline{3-gon outer facets are allowed}.  Sometime we also allow \underline{two outer facets to share edges} in a semi-MPG. As a 3-gon outer facet, it does not need to follow Definition~\ref{def:Rtiling} for any provided R-tiling $T_r$. As for edge counting, a shared edge shall be counted by multiplicity 2 w.r.t\ the two outer facets that it belong to.

Our interest is focused on MPG's and semi-MPG's. In the last section, we saw several examples that an $(5,7)$-semi-MPG might cause problems, and then we realized that an MPG and an $n$-semi-MPG are good to to deal with. For fun we call an MPG as well as an $n$-semi-MPG \emph{One Piece}: a Japanese manga series), because they really are a solid planar piece, where ``solid'' means no hole. 

The three (counter) examples in the last section all provided a non-one-piece which looks like an eye of volcanic island or a belt of ring. Even though a belt of ring or an $(n_1,n_2,\ldots,n_k)$-semi-MPG is planar and definitely 4-colorable, but its topological structure creates obstacles for our renewal approach, i.e., sometimes we fail to translate a single color tiling on an $(n_1,n_2,\ldots,n_k)$-semi-MPG to a 4-coloring function. To fix this problem, we shall study some extra requirement set for a ``good'' single color on an $(n_1,n_2,\ldots,n_k)$-semi-MPG.

One Piece is definitely the main topic of further discussion. However, here is the last section that we care about an $(n_1,n_2,\ldots,n_k)$-semi-MPG $M$. Once an R-tiling $T_r$ exists on an $(n_1,n_2,\ldots,n_k)$-semi-MPG $M$, there are many red-connected components $rC_i$ for $i=1,2,\ldots$ as induced subgraphs of $M$; even if $rC_i$ is just a single point. We can also define  $T_r:=\bigcup_i E(rC_i)$ as a set of red edges. If we just consider a short line of red edges in any $rC_i$, they forms a wall of red \emph{canal bank} (of river bank). For instance, $v_6$-$v_0$-$v_1$-$v_2$-$v_e$-$v_2$-$v_f$-$v_2$ in Figure~\ref{fig:deja_vu} is a (part of) canal bank.  We shall define \emph{canal lines} first and then explain canal banks precisely.

   \begin{figure}[h]
   \begin{center}
   \includegraphics[scale=1.2]{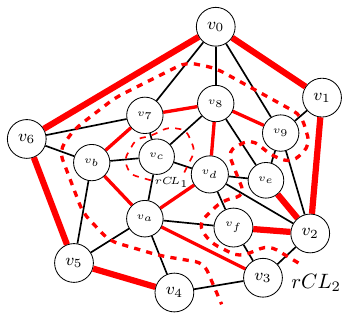}
   \end{center}
\caption{Canal lines and canal banks; Four deja-vu edges for $rCL_2$}  
\label{fig:deja_vu}     
   \end{figure}    
    \begin{definition}  \label{def:deja_vu} 
Let $M$ be a MPG or semi-MPG with an R-tiling $T_r=\bigcup_i E(rC_i)$. We use Figure~\ref{fig:deja_vu} to demonstrate examples. Among and inside these $rC_i$ are red \emph{canal lines}, denoted by $rCL_j$ for $i=1,2,\ldots$. For example, there are two red canal lines $rCL_1$ and $rCL_2$ represented by dashed lines . The two sides of any $rCL$ are its coexisting \emph{left/right canal bank}, denoted by $rCL^r$/$rCL^l$ (or $rCB_j$ as general notation). For example, if we consider $rCL_1$ flow clockwise then $rCL_1^r=v_e$ and $rCL_1^l=v_0$-$v_1$-$\ldots$-$v_6$-$v_0$. Following the current along a red canal line $rCL$, we might see a red edge twice from its different sides. We call it a \emph{deja-vu} edge for $rCL$. In Figure~\ref{fig:deja_vu}, there is no deja-vu edge for $rCL_1$ and there are four deja-vu edges for $rCL_2$, namely $v_3v_a$, $v_8v_9$, $v_2v_e$ and $v_2v_f$.
    \end{definition}

In manga and comic style, the ships of cargo and pirate follow the one-way \emph{current} of a red canal line. The basic or elemental segment of a red canal line is a single red half-tile (triangle) who has one red edge and two black edges. A complete red canal line is made by a sequence of red half-tiles, while the current flow into and out of a triangle by crossing two black edges. 

The collect of all these canal lines $rCL_i$ created by an fixed R-tiling on $M$ is called a red \emph{canal system}, denote and defined by $rCLS=\bigcup_j rCL_j$. We shall treat canal system first as a normal graph and then a directed graph. While treated as a normal graph, every triangle of red half-tile, which is a triangle, plays as a node of degree two, where the two links\footnote{We shall avoid to use the name edge/vertex for the canal system $rCLS$ which is actually a subgraph of the dual graph of $M$.} are translated from the two black edges of this triangle.  A red canal line $rCL_j$ is than a collection of connected links. Because every node is degree two, each canal line $rCL_j$ is either a cycle or a path, where the latter one is from one outer facet to another facet (maybe the same outer facet). A canal system must have the following properties.

   \begin{lemma}   \label{thm:InOutCanal}
Let $M$ be an MPG or an $n$-/$(n_1,n_2,\ldots,n_k)$-semi-MPG with an R-tiling $T_r$.  
      \begin{enumerate}
\item[(a)] A red tiling $T_r:=\bigcup_i rC_i$ on $M$ and a red canal system $rCLS:=\bigcup_j rCL_j$ are different perspectives of looking the same thing.
\item[(b)] By linking nodes of triangles, a red canal line $rCL_j$ of $T_r$ is either {\rm (b1)} a close cycle, called \emph{canal ring}, or {\rm (b2)} a path starting from one outer facet and ending at another outer facet (maybe the same outer facet), while the pair of entrance and exit on the two end of this path are both black edges along the outer facets. 
\item[(c)] If $M$ is an MPG, then every red canal line $rCL_j$ is ring. If $M$ is an $n$-semi-MPG, then the connection of entrances and exits
of this red canal system $rCLS$ creates a non-crossing match among all black edges along the unique outer facet.   
      \end{enumerate}      
   \end{lemma}
   \begin{proof}
The proofs of (a), (b) and the first part of (c) have been already stated in the previous paragraphs. We need only provide a proof for the second part of (c). The reason is very simple: a $n$-semi-MPG is planar.  
   \end{proof}

Now we need to indicate the \emph{direction of current} for each red canal line, i.e., we see canal system as a directed graph. While treated as a directed graph, 
a single triangle as a node is incident with two directed links which are one in and one out. Here is the rule of choosing direction: If a single color R-tiling on an $(n_1,n_2,\ldots,n_k)$-semi-MPG associates with a 4-coloring, along the current of any red canal line we shall always see color-1 vertices laid on the canal bank of our right hand side. So the perfect currents for a red canal system $rCLS$ has a important requirement: For every red diamond, the directions of two currents on the two sides of this red edge must be opposite.    

As for a G-tiling or a B-tiling on an $(n_1,n_2,\ldots,n_k)$-semi-MPG, we still obey the rule of choosing direction: ``color-1 vertices laid on the canal bank of our right hand side'' for $GCS$ or $BCS$.  

    \begin{example}[Conquering Counter-Examples~\ref{ex:Counterexample2} and~\ref{ex:Counterexample3}]  \label{ex:Conquer}
According to the hypothesis of Theorem~\ref{thm:4RGBtiling}, we had better focus on an MPG or an $n$-semi-MPG. However, here we still try to conquer the last two counterexamples without involving Theorem~\ref{thm:4RGBtiling}. For Example~\ref{ex:Counterexample2}, we keep most red edges, but let $v_6v_7$ black and let $v_0v_6$, $v_7v_b$ red. This is edge-color-switching along a diamond route. The idea of diamond routes will be formally introduced in the Part III of our paper. The new R-tiling is shown by the left graph below and it is a ``good'' R-tiling. 
   \begin{figure}[h]
   \begin{center}
   \begin{tabular}{ c c }
   \includegraphics[scale=0.79]{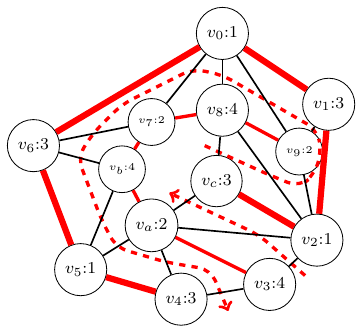}\
   \includegraphics[scale=0.9]{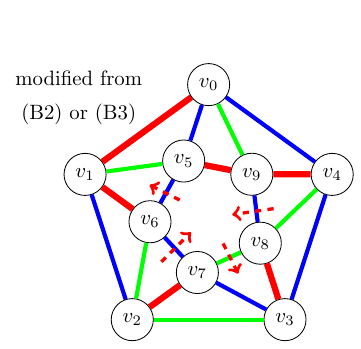}
   \end{tabular}
   \end{center}
\caption{Countering the errors in Examples~\ref{ex:Counterexample2} and~\ref{ex:Counterexample3}}  
\label{fig:Conquer}     
   \end{figure}
As for Example~\ref{ex:Counterexample3}, we shall back to Figure~\ref{fig:counterEX31} and consider the right two graphs. In these two graphs, we left some dashed-lines in advance. These dashed-lines are canal lines according to their own colors.  Along these red and green dashed-lines, we perform edge-color-switching and then obtain the right graph in Figure~\ref{fig:Conquer} from both (B2) and (B3).  This new RGB-tiling is ``good'' enough to induce a 4-colorable function. The way to perform edge-color-switching along an R/G/B canal lines will be will be formally introduced in the Part III. One more thing to do is indicating the directions of current system by the four red dashed-lines. We made these four directions perfect by using two ins and two outs along the 5-gon alternately. In this way, the two directions on the two sides of every red edge in diamond are opposite.   
   \end{example}

With the experience from Examples~\ref{ex:Counterexample1}, \ref{ex:Counterexample2}, \ref{ex:Counterexample3} and~\ref{ex:Conquer}, we claim the following definition to make clear about ``good'' R-tilings. Here we define a \emph{grand} R-tiling in view of red canal banks and a \emph{grand} R-canal system in view of orientation of currents. Actually they are same concept from two points of view.
   \begin{definition}  \label{def:grandGrand}
Let $M$ be an $(n_1,n_2,\ldots,n_k)$-semi-MPG with an R-tiling $T_r$. Here we \underline{allow} two outer facets to share edges. Let $M_r$ denote the subgraph of $M$ by deleting all black edges and $M_{bl}$  denote the subgraph of $M$ by deleting all red edges. We say this R-tiling $T_r$ is a \emph{grand} tiling if $V(M)$ can be partitioned into two disjoint parts $V_{13}$ and $V_{24}$, i.e., $V(M) = V_{13} \uplus V_{24}$, such that $M_{bl}$ is a bipartite graph with bipartite vertex sets $V_{13}$ and $V_{24}$, and no red edges link between $V_{13}$ and $V_{24}$. 
   \end{definition} 
   \begin{lemma}  \label{thm:RandomWalk}
Let $M$ be an $(n_1,n_2,\ldots,n_k)$-semi-MPG. If $M$ has a grand R-tiling $T_r$, then the two subgraphs induced in $M$ by $V_{13}$ and $V_{24}$ consist of all red edges of $T_r$ and no black edges. Furthermore, for any vertices $w, x\in V_{13}$ and $y, z\in V_{24}$, any walk through black edges from $x$ to $y$ must be odd; and any walk through black edges from $w$ to $x$ or from $y$ to $z$ must be even.      
   \end{lemma}   

Why do we use notation $V_{13}$ and $V_{24}$? Please, refer to Figure~\ref{fig:1234EdgeColor}.   
   
   \begin{definition}  \label{def:grandGrand2}
We say the red canal system $rCLS$ induced by an R-tiling on $M$ is a \emph{grand} canal system if we can arrange orientation for all canal lines such that the flow directions are opposite on the two sides of each red edge in diamond. Notice that here we \underline{do not allow} two outer facets to share edges. 
   \end{definition} 

For $M$ without any edge shared by two outer facets, these two definitions are equivalent by our general rule: when we follow the direction of red current $rCL$ (a red canal line with a fixed direction), the red canal bank on the right-hand-side always belongs to $V_{13}$ which contains vertex-color 1 and denoted by $rCL^r$ (or $rCB_i$ as general notation), and the red canal bank on the left-hand-side, denoted by $rCL^l$, always belongs to $V_{24}$ that has no vertex-color 1. Sometimes we will draw red edges in $V_{13}$ thicker than those in $V_{24}$, just like the right two graphs in Figure~\ref{fig:1234EdgeColor} and the right graph in Figure~\ref{fig:twoRGBtiling}. Given a grand R-tiling $T_r$ on $M$, the red-connected components $rC_i$ are sorted into two classes: some belong to the subgraph induced in $M$ by $V_{13}$, and the others belong to the one by $V_{24}$.

   \begin{lemma}    \label{thm:grandRtilingGrandLineSys}
Given an R-tiling $T_r$ on an MPG or an $n-$/$(n_1,n_2,\ldots,n_k)$-semi-MPG, it is grand if and only if the red canal system $rCLS$ induced by R-tiling is grand.  
   \end{lemma}

Once an R-tiling is grand, we can set $V_{13}$ ready for colors 1 and 3, and $V_{24}$ ready for colors 2 and 4. However, ``ready'' does not mean a real 4-coloring function, we still need the crucial requirement: each red-connected component $rC_i$ has no red odd-cycle. 

Once we stitch to focus on a G-tiling/B-tiling, we shall follow the general rule: those vertex-color 1 are always on the right-hand-side if we follow the green/blue current. Therefore, we see $V_{14}$/$V_{12}$ as canal banks on the right-hand-side of all green/blue currents.

   \begin{lemma}    \label{thm:grandRtilingNoOddCycle}
Given an R-tiling on an MPG or an $n$-/$(n_1,n_2,\ldots,n_k)$-semi-MPG, say $M$,, if it is grand and has no red odd-cycles, then it can induce a 4-coloring function on $M$.  
   \end{lemma}

The next theorem tells us why One Piece is so important.

   \begin{theorem}[Theorem for One Piece]    \label{thm:RtilingOnePiece}
Every R-tiling on One Piece (which is either an MPG or an $n$-semi-MPG, $n\ge 3$) must be grand one.
   \end{theorem}
   \begin{proof}
Denote this One Piece by $M$ and it has an R-tiling $T_r$. All we need to show is that [CLAIM] $V(M)$ can be partitioned into two disjoint parts $V_{13}$ and $V_{24}$ such that the subgraph $M_{bl}$ induced by black edges of $T_r$ is a bipartite graph with bipartite vertex sets $V_{13}$ and $V_{24}$, and no red edges link between $V_{13}$ and $V_{24}$. We will prove it by induction.
   
[A: $T_r$ on an $n$-semi-MPG]: {\bf Case A1.} Suppose all $n$ edges along the outer facet $n$-gon are red. Notice that when we have a border of $3$-gon outer facet, its three edges do not need to follow Definition~\ref{def:Rtiling}.  

Without loss of generality, we set these $n$ red edges along the unique out facet to be a part of $V_{13}$.  Let Figure~\ref{fig:OnePieceProof} be a demonstration, where $M$ has a 8-gon outer facet $v_0$-$v_1$-$\ldots$-$v_7$-$v_0$. We also denote $rC_1$ to be the red-connected component that contains these $n$ red edges. The red edges of $rC_1$ are exactly thick ones in Figure~\ref{fig:OnePieceProof}. 

Let us choose any red edge $e$ (say $v_0v_7$) along this outer facet, and then we follow the red canal line $rCL$ generated by $e$-triangle (there is no $e$-diamond) and indicate the direction of current to make $e$ on the right hand side because $v_0, v_7 \in V_{13}$.
   \begin{figure}[h]
   \begin{center}
   \includegraphics[scale=1]{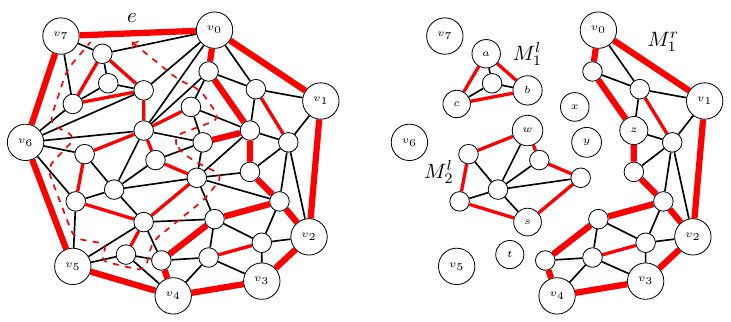}
   \end{center}
\caption{Proving process of Case A1}  
\label{fig:OnePieceProof}     
   \end{figure}
This $rCL$, the red dashed line with direction in Figure~\ref{fig:OnePieceProof}, must be a ring because all $n$ edges along the outer facet are red.  Obviously all red edges of canal bank on the right of $rCL$, denoted the edge set by $rCL^r$,  belong to $rC_1$. And then we can name $rC_2$ to be the red-connected component that contains those red edges of canal bank on the left of $rCL$, denoted the edge set by $rCL^l$. In Figure~\ref{fig:OnePieceProof}, $rCL^r$ and $rCL^l$, separated by $rCL$, are definitely red-disconnected; they are kind of parallel and linked by black edges as zigzag line drawing. (PS. This is not always true if the surface is not planar. The surface of torus can provide a counterexample.) 

How to guarantee [CLAIM]? At this stage (or every new stage), we only consider those vertices and edges involved by $rCL$ (respectively $rCL_i$). We see the right and the left canal banks of $rCL$ are two groups of vertices, say $V(rCL^r)$ and $V(rCL^l)$, which are subsets of the final $V_{13}$ and $V_{24}$ respectively. If we deal with $rCL_i$  at some following stage, exactly one of $V(rCL_i^r)$ and $V(rCL_i^l)$ \underline{inherits the mark} obtained from either $V_{13}$ or $V_{24}$; it is at least one due to connected graph; it is only one due to $M$ being planar graph and One Piece.\footnote{Please see the left graph in Figure~\ref{fig:counterEX3}. It is a counter example for an $(n_1, n_2)$-semi-MPG in Case A2. The reader might try a  counter example for an $(n_1, n_2)$-semi-MPG in Case A1.} Two groups of vertices in $V(rCL^r)$ and $V(rCL^l)$, separated by $rCL$, only use black edges to link between. Also those edges linking among vertices in $V(rCL^r)$ (or in $V(rCL^l)$) are red edges as the wall of canal bank of $rCL$ at this stage. Therefore, the [CLAIM] is good for $\bar{V} :=V(rCL^r)\ cup V(rCL^l)$ and those red edges and black edges involved by $rCL$. Not only taking care themselves, but also $V(rCL^r)$ and $V(rCL^l)$ inherit their marks obtained from $V_{13}$ and $V_{24}$ and bring the marks to the following stages of our induction proof; because on the other side of every non-deja-vu and non-outer-facet red edge, we will deal with the other $rCL_i$. We shall keep adding new vertices into $\bar{V}$ as bipartite-way and check the behavior of those red edges and black edges involved by the new $rCL_i$. While we obtain $V(M)=\bar{V}$, the proof of [CLAIM] is done.

Now we need to do something as preparation for the coming stages. We shall
    \begin{enumerate}
\item[(Blk):] delete all black edges along $rCL$, and 
\item[(Red):] delete all deja-vu edges along $rCL$, as well as the red edges that are both on the right of $rCL$ and along the $n$-gon.    
    \end{enumerate}
For our example, we shall delete (Red) four deja-vu edges, namely $bw$, $wx$, $yz$ and $st$. Also delete $v_4v_5$, $v_5v_6$, $v_6v_7$ and $v_7v_0$.  The remaining subgraph consist of several parts as red/black-connected components. Actually every part must be red-2-connected for its border. Being only red-1-connected, they shall be treated as different parts rather then one part. (i) The most easy part is a single vertex incident to no edges. For example, $v_5$, $v_6$, $v_7$, $x$, $y$ and $t$. At this new stage, a single vertex is the base case in our induction proof. This single vertex belongs to either $V(rCL^r)$ or $V(rCL^l)$ is easy to judge by $rCL$. (ii) Some parts are on the \underline{right} of $rCL$ each of which are again $n_i$-semi-MPG with all $n_i$ edges along the outer facet $n_i$-gon being red. Let us denote these parts by $M^r_i$. Some $M^r_i$ and $M^r_j$ might be red-1-connected; so we shall treat them as two distinct $n_i$-gon and $n_j$-gon outer facets as well as $n_i$- and $n_j$-semi-MPG. For our example, we only have $M^r_1$ which has two big red cycles that are red-2-connected. The out facet of each $M^r_i$ contains some vertices of the right canal bank of $rCL$ (or $V(rCL^r)$); therefore, all $n_i$ vertices of this out facet inherit the mark from $V_{13}$ made by the previous stage of our induction. Obvious, let us redo the process of Case A1 on these $M^r_i$. (iii) Some parts are on the \underline{left} of $rCL$, and let us  mark them by $M^l_j$. In Figure~\ref{fig:OnePieceProof}, we have $M^l_1$ and  $M^l_2$. (Let us further think about a new situation: if we merge vertices $b$ and $w$, then $M^l_1$ and  $M^l_2$ are red-1-connected; so they shall still be treated as two independent semi-MPGs.) The argument are nearly the same as (ii), unless all vertices of these out facets shall inherit the mark from $V_{24}$ made by the previous stage of our induction.

Keep working on (ii) or (iii) by redoing the process of Case A1 until all threads are terminated by (i). The process and [CLAIM] will be done finally.

{\bf Case A2.} Suppose there are $2k$ black edges ($n\ge 2k\ge 0$) along this single outer facet. This number must be even due to Lemma~\ref{thm:evenoddRGB}(a). We are going to apply induction proof for Case A2 by reducing this index $k$. Notice that $k=0$ is exactly Case A1 and it plays a role as the base case of our induction proof in this portion. The argument of the proof of Case A2 is similar to the proof of Case A1; so, we just sketch the main ideas of this proof. Especially how to guarantee [CLAIM]? Just follow the argument given in Case A1. 

Referring to Lemma~\ref{thm:InOutCanal}(c), all black edges along the unique outer facet form a non-crossing match by 
the connection of entrances and exits of this red canal system of $T_r$. Let us just choose a pair of black edges of this match, namely $e_1$ and $e_2$ in Figure~\ref{fig:OnePieceProof2} as a demonstration. 
   \begin{figure}[h]
   \begin{center}
   \includegraphics[scale=1]{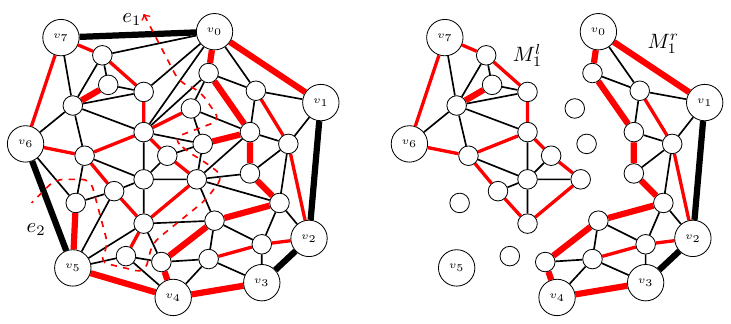}
   \end{center}
\caption{Proving process of Case A2}  
\label{fig:OnePieceProof2}     
   \end{figure}
A directed current (red canal line) $rCL$ is immediately provided. Again, we shall delete  (Blk) and (Red), in addition we shall delete $e_1$ and $e_2$. Like Case A1, we now have three kinds of parts as red/black-connected components: (i) a single vertex incident to no edges (ii) $M^r_i$ (iii) $M^l_j$. The argument of induction to show the [CLAIM] true is nearly the same. The key of induction is that $e_1$ and $e_2$ have been deleted and then each $M^r_i$ and $M^l_j$ has at most $2k-2$ black edges along its unique outer facet. We can perform this induction proof with this important key. Now we finish sketching the main ideas of this proof.

[B: $T_r$ on an MPG]: Every $rCL_i$ formed by this $T_r$ is a ring. We simply choose one $rCL$ and perform the same process given in Case A1, then we obtain three kinds of parts as red/black-connected components: (i) a single vertex incident to no edges (ii) $M^r_i$ (iii) $M^l_j$. The key of induction is that each $M^r_i$ and $M^l_j$ has its outer facet all red edges, and we shall directly apply Case A1 on each $M^r_i$ and $M^l_j$.  Now we finish sketching the main ideas of the proof of this portion.   
   \end{proof}

It is the right time to cash in the previous rain check. 
   \begin{proof}\
{\bf[(c) $\Rightarrow$ (a) for Theorem~\ref{thm:4RGBtiling}]:} An R-tiling on an  $M$, an MPG or an $n$-semi-MPG (One Piece), must be grand one by Theorem~\ref{thm:RtilingOnePiece}. So we have $V(M)=V_{13}  \uplus V_{24}$ and  $V_{13}$ is ready to color 1 and 3 and $V_{24}$ is ready to color 2 and 4. No red odd-cycle in $T_r$ guarantees a 4-coloring must be done.
   \end{proof}
   
Now we give a generalization of Theorem~\ref{thm:RtilingOnePiece}.
   \begin{lemma}    \label{thm:evenNumberBlack}
Let $M$ be an MPG, $k$-semi-MPG, or an $(k_1,k_2,\ldots,k_t)$-semi-MPG with an R-tiling $T_r$. Here we \underline{allow} two outer facets to share edges. The following three are equivalent: 
   \begin{itemize}
\item[(a)] The number of black edges along any cycle in $M$ is always even.
\item[(b)] The number of black edges along any outer face in $M$ is always even.
\item[(c)] The red tiling $T_r$ is grand. (Definition~\ref{def:grandGrand}, not Definition~\ref{def:grandGrand2}.)
   \end{itemize}
   \end{lemma}
   \begin{proof}\  
[(a) $\Leftrightarrow$ (b)]: Item (b) is weaker than (a), thus ``(a) $\Rightarrow$ (b)'' is trivial. Let us prove ``(a) $\Leftarrow$ (b).'' Provided any cycle $C$ that is not an outer facet in $M$, then inside of $C$ there are some outer facets or probably none, i.e., the structure inside of $C$ is then another MPG or another $(m_1,m_2,\ldots,m_h)$-semi-MPG, denoted by $M_C$, where the cycle $C$ itself forms a outer facet in $M_C$. By Lemma~\ref{thm:evenoddRGB}(a'), the total number of black edges in $\Omega(M_C)$ is even. As (b) provided, we shall see even number of the total black edges along the original outer facets of $M$ inside $C$, and then we conclude that $C$ must have even number of black edges.

[(a) $\Leftarrow$ (c)]: Given (c), the subgraph $M_{bl}$, which is induced by black edges of $T_r$, is bipartite with bipartite vertex set $V_{13}$ and $V_{24}$; also there is no red edge linking between $V_{13}$ and $V_{24}$. Given any cycle with red and black edges. If this cycle is made by pure black edges, then definitely it is even length. Suppose there are some red edges along this cycle. We simply remove these red edges which only links vertices in  
$V_{13}$/$V_{24}$. So the total number of black edges zigzagging between $V_{13}$ and $V_{24}$ is even.

[(a) $\Rightarrow$ (c)]: Given (a), we need to show $M_{bl}$, the subgraph induced by black edges of $T_r$, is bipartite. There is an elementary equivalent condition for a bipartite graph: A graph is bipartite if and only if every cycle of the graph is even length, including the situation that there is no cycle at all. The statement (a) do provide all possible black cycles must be even length; hence $M_{bl}$ is a bipartite graph. To finish this proof we still need to show that any red edge $ab$ of $T_r$ only links vertices in the same partite set.  

(A): If we do not allow two outer facets to share edges, then $M_{bl}$ must be black-connected; because removing any red edge $ab$ from $M$, vertices $a$ and $b$ are still connected by at least two black edges belonging to red triangle or red half-tile of $ab$. For black-connectivity, the bipartite sets $V_{13}$ and $V_{24}$ of $M_{bl}$ has only one way, while switching $13$ and $24$ for whole $V(M)$ is symmetric. In this case, the end vertices $a$ and $b$ of any red edge $ab$ must either belong to $V_{13}$ or $V_{24}$. The reason still comes from the two black edges of the half-tile of $ab$.  

(B): However, if we allow two outer facets to share edges, then $M_{bl}$ might be black-disconnected. Here we offer two different ways, (B1) and (B2), to conquer this situation. (B1) For any shared edge $ab$ with red color, we simply draw two extra black edges to make a new red half-tile of $ab$. Now the hypothesis of (a) still holds for this new $(k'_1,k'_2,\ldots,k'_t)$-semi-MPG $M'$ with a new R-tiling $T'_r$, where $M'$ has no shared edges of two outer facets. Therefore, (A) guarantees $M'$ correct for (c), and then $M'$ guarantees $M$ correct for (c) after removing all extra black edges. Then the proof is complete. 

(B2) Removing those shared red edges of two outer facets, denoted this edge set by $SRE$, we might get a disconnected subgraph graph, say $\bar{M}_0$, with components $T_1, T_2,\ldots, T_p$. Especially each $T_i$ is black-connected and it is correct for (c) with two partite sets of its vertices. Since $M$ is connected, there is a subset $SRE'$ of $SRE$ such that restoring all red edges in $SRE'$ back to $\bar{M}_1$ will link all $T_1, T_2,\ldots, T_p$ to be 1-connected, i.e., if we treat each $T_i$ as a single node then $SRE'$ induces a tree. Now we can unify all pairs of bipartite sets of $T_i$ to associate with their own family, $V_{13}$ or $V_{24}$, through this connection by $SRE'$; also all red edges so far definitely links vertices in $V_{13}$/$V_{24}$.  How about the rest red edges in $SRE-SRE'$? We shall start with $\bar{M}_1$ and restore these red edges one-by-one. Once we return a red edge $ab\in SRE-SRE'$ back to $\bar{M}_i$ to reach the current stage $\bar{M}_{i+1}$, all cycles passing through $ab$ at this current stage have even number of black edges by (a); obviously both $a$ and $b$ belong to either $V_{13}$ or $V_{24}$ at the previous stage $\bar{M}_i$; so this red edge $ab$ satisfies (c). Notice that this returning method has to be one-by-one; for instance, we cannot deal with $ab, a'b'\in \bar{M}_{i+1}-\bar{M}_{i}$, because a cycle passing through both $ab$ and $a'b'$ will cause a nightmare. Once $\bar{M}_j=M$, the proof of (B2) is complete.      
   \end{proof}

   \begin{remark}
Lemma~\ref{thm:evenoddRGB}(a) and (a') check the parity of the total number of black edges along all outer facets of a semi-MPG with an R-tiling, and Lemma~\ref{thm:evenNumberBlack}(b) also exams all outer facets, but check them one by one.
   \end{remark}

   \begin{theorem}[The First Fundamental Theorem v2: a generalized version of Theorem~\ref{thm:4RGBtiling}]  \label{thm:4RGBtilingGeneralization}
Let $M$ be an MPG or an $n$-/$(n_1,n_2,\ldots,n_k)$-semi-MPG. Here we allow two outer facets to share edges in $M$. The following are equivalent:
   \begin{itemize}
\item[(a)] $M$ is 4-colorable.
\item[(b)] $M$ has an RGB-tiling such that along every $m$-cycle in $M$ the numbers of red, green and blue edges are all even if $m$ is even, and all odd if $m$ is odd.
\item[(c)] $M$ has an RGB-tiling such that along every $n_i$-gon outer facet the numbers of red, green and blue edges are all even if $n_i$ is even, and all odd if $n_i$ is odd.
\item[(d)] $M$ has a grand R-tiling without red odd-cycles.
   \end{itemize}
   \end{theorem}
   \begin{proof}\
[(b) $\Leftrightarrow$ (c)]: An RGB-tiling is also three coexisting R-, G- and B-tilings. We just apply Lemma~\ref{thm:evenNumberBlack}(a) and (c) w.r.t.\ R-, G- and B-tilings, then we obtain ``(b) $\Leftrightarrow$ (c)'' of this Theorem.

[(a) $\Rightarrow$ (b)]: Let $f:V(M)\rightarrow \{1,2,3,4\}$ be a 4-colorable function and $C:= v_1$-$v_2$-$\ldots$-$v_m$ be any $m$-cycle in $M$. For symmetry we only show a general claim that the total number of red edges and blue edges along $C$ is even.  The function $f$ restricted on $C$ can be traced as a corresponding closed walk around the middle graph in Figure~\ref{fig:1234EdgeColor}, i.e., we just follow the numbers $f(v_i)|_{v_i\in C}$. Since we only consider red edges and blue edges, which are the two vertical lines and the two diagonal lines of the middle graph in Figure~\ref{fig:1234EdgeColor}; but we ignore the two horizontal lines. Let us consider an invisible horizontal line $\ell$ which just lays in the middle of ``the middle graph in Figure~\ref{fig:1234EdgeColor}.''  To be a closed walk, definitely there are as many up direction crossings of $\ell$ (including ``up'' vertically and ``up'' diagonally) as down direction crossings of $\ell$. Therefore, the total number of red edges and blue edges along the closed walk as well as along $C$ must be even.
Thus the proof is done. 

[(b) $\Rightarrow$ (d)]: Let $T_r$ be an R-tiling which is induced by the given RGB-tiling on $M$ claimed in (b). By the equivalence given in Lemma~\ref{thm:evenNumberBlack}, $T_r$ is a grand R-tiling when we treat green edges and blue edges as black color. If there is a red cycle in $M$, the cycle must be even; because in this red cycle the numbers of green edges and blue edges are both zero.    
   
[(d) $\Rightarrow$ (a)]:  We already have had Lemma~\ref{thm:grandRtilingNoOddCycle} to prove this direction.
   \end{proof}

Let us think differently. Suppose that finding a 4-coloring function is not our primary purpose, but we do need to know whether this $R$-tiling can induce a coexisting G- or B-tiling or not.  The coexistence property for whole $M$ can be divided into small piece to check, i.e., we shall check each red canal line $rCL_i$. The black edges zigzagging along $rCL_i$ are supposed to color green and blue alternately. There is no problem for a $rCL_i$ as a line; however for a $rCL_i$ as a ring, we must require the number of triangles along $rCL_i$ even.   

Furthermore, let us study the parity property that connects between those triangles along $rCL_i$ and those cycles of red canal bank along $rCL_i$.

   \begin{lemma} \label{thm:triangles}
Let $M$ be an MPG or a $k$-/$(k_1,k_2,\ldots,k_t)$-semi-MPG with an R-tiling $T_r$. Let us choose any red canal line $rCL$ as a \underline{ring} if there is. Recall the notation $M^r_i$ and $M^l_j$ defined in the proof of Theorem~\ref{thm:RtilingOnePiece}; but this time (Red): only delete deja-vu edges along $rCL$.
   \begin{enumerate}
\item[(a)] All $M^r_i$/$M^l_j$ are semi-MPG's. 
\item[(b)] We are only interested in those $n_i$-/$m_j$-gon outer facets of $M^r_i$/$M^l_j$ along $rCL$. The number of triangles along $rCL$ is even if and only if the sum $\sum_i n_i +\sum_j m_j$ is even. 
   \end{enumerate}
   \end{lemma}
   \begin{proof}
We would like to use Figure~\ref{fig:OnePieceProof3} as an example. After deletion by (Blk) and (Red), we obtain $M^r_1$, $M^l_1$ and $M^l_2$ respectively with 11-, 3- 
and 6-gons, which are marked by black dots; so $\sum_i n_i +\sum_j m_j=20$. Notice that $M^r_1$ is an $(11,9)$-semi-MPG and we are only interested in its 11-gon.  There are 28 triangles along $rCL$. Their difference is 8 and it is caused by 4 deja-vu edges.
   \begin{figure}[h]
   \begin{center}
   \includegraphics[scale=1]{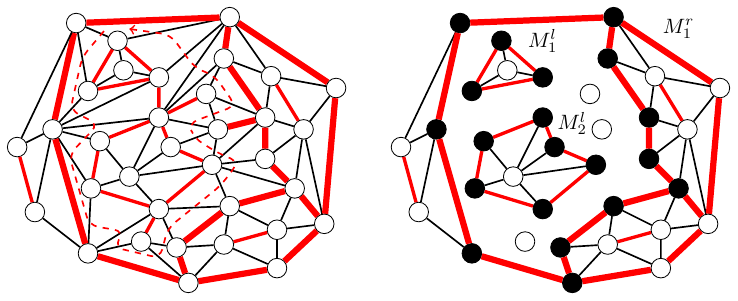}
   \end{center}
\caption{Process of Lemma~\ref{thm:triangles}}  
\label{fig:OnePieceProof3}     
   \end{figure}

Here we need a general identity good for both ring or path:  
  $$
\text{\#(triangles along $rCL$)}= e(rCL^r) + e(rCL^l),
  $$    
where $e(rCL^\ast)$ is the number of red edges along the canal bank $rCL^\ast$ with each deja-vu edge counted by multiplicity 2. Because each deja-vu edge counted by multiplicity 2, so the parity of $e(rCL^r) + e(rCL^l)$ equals the parity of $\sum_i n_i +\sum_j m_j$.  
   \end{proof}

In a general graph, a circuit or closed walk (of odd length) made by edge set $W$ of sequence, then there exists a subsequence $C\subseteq W$ such that $C$ is a minimal cycle (of odd length respectively). Let us back to Figure~\ref{fig:OnePieceProof3}, we are interested in 11-, 3- and 6-gons of red cycles with vertices marked by black dots. Precisely we consider them the \emph{minimal} red cycles w.r.t.\ to this $rCL$. Literally a minimal red cycle means no way to shrink it smaller. However, what is ``smaller''? An MPG is kind of sphere and there are two sides to make a cycle smaller. That is why we need to add ``w.r.t.\ to this $rCL$''. 

   \begin{example}
In Figure~\ref{fig:OnePieceProof3} for instance, the 3-gon/6-gon red cycle is minimal not only w.r.t.\ to this $rCL$ but also w.r.t.\ to another canal line inside itself. Let us look at $M^r_1$ in Figure~\ref{fig:OnePieceProof3}.  There are seven different red cycles (there are eight different red cycles if including the cycle of empty). Obviously 3 of them are minimal ones, namely the 11-gon, 7-gon and 6-gon are minimal w.r.t.\ their own $rCL_i$. The rest four cycles are made by combinations of these three minimal ones.    
   \end{example}

Having all red cycles of $T_r$ even length is the best. In this way we can achieve a coexisting RGB-tiling. However, this RGB-tiling might not be grand for its R-, G- and B-tilings. Such kind of examples have been shown in Example~\ref{ex:Counterexample3}.  Anyway, once an RGB-tiling exists, there are $2^N$ different coexisting RGB-tiling induced by this R-tiling, where $N$ is the total number of red canal lines $rCL_i$ including both rings and paths. 

   \begin{remark}
In order to get an RGB-tiling $T_{rgb}$ on an MPG or an $n$-/$(n_1,\ldots,n_k)$-semi-MPG $M$, we shall first try to pave an R-tiling $T_r$. Suppose that occasionally we obtain a red odd-cycle $C_r$ in the middle of our working. Let $\Sigma^+$ ($\Sigma^-$) denote   
the subgraph of $M$ inside (outside) of $C_r$. \underline{If $\Sigma^+$ is an $|C_r|$-semi-MPG} with the only outer facet $C_r$, then we might achieve an R-tiling $T_r$ on $\Sigma^+$, but never an RGB-tiling.  Why? Let us consider the last step to achieve a possible $T_r$ on $\Sigma^+$. This last step can be two possible ways: (1) to pave a red diamond inside $C_r$ or (2) to pave a red triangle along $C_r$. Just before this last step to pave, we see a $(4, |C_r|)$-semi-MPG $M_1$ for (1); or a $(3, |C_r|)$-semi-MPG $M_2$ for (2), where outer facets 3-gon and $|C_r|$-gon share a same red edge. Now we shall refer to Corollary~\ref{thm:evenoddRGB2} and Remark~\ref{re:CSigma}.   The all-even/all-odd property cannot hold for $M_1$ and $M_2$. We will see $M=EP$ and both $\Sigma^+$ and $\Sigma^-$ in form of $M_2$ in Part II of this paper.
   \end{remark}

\section{Degree 5 vertices in any $EP$}

Referring to Theorem~\ref{thm:V5more} and $EP\in e\mathcal{MPGN} \subseteq m\mathcal{N}4$, we know any $v\in V(EP)$ must have $\deg(v)\ge 5$. An important preliminary knowledge is about vertices in $EP\in e\mathcal{MPGN}4$ being degree exactly 5. 

Given a graph $G$, let $V_k(G)$ (or simply $V_k$) denote the set of vertices of degree $k$ in $G$. Also let $\#V_k(G)$ (or simply $\#V_k$) denote cardinality of $V_k(G)$. If $G$ is a planar graph, the general notation $\#V, \#E, \#F$ are the numbers of vertices, edges, facets in $G$ respectively.
   \begin{theorem} \label{thm:V5}
Let $G$ be an MPG with the degrees of all vertices at least 5. We have an inequality
   \begin{equation}
   \#V_5 = 12+ \#V_7 +2\#V_8 +3\#V_7 +\cdots  \label{eq:V5}
   \end{equation}
   Or there are at least 12 vertices are exactly degree 5.         
   \end{theorem}
   \begin{proof}
Let us first ignore the requirement that degree at least 5. We use the minimum requirement: all vertices are degree at least 2 to preserve a good looking for triangular facets of $G$. We consider four equations: $\#V -\#E +\#F=2$ (Euler's formula), $3\#F=2\#E$ (every facet is a triangle), and $\Sigma_{k\ge 2} k \#V_k =2\#E$, and $\#V=\Sigma_{k\ge 2} \#V_k$.
   \begin{eqnarray}
   6\#V-(2+4)\#E+6\#F &=& 12; \nonumber \\  
   6\Sigma_{k\ge 2} \#V_k -\left(\Sigma_{k\ge 2} k\#V_k + 6\#F \right) +6\#F &=& 12;\nonumber \\ 
   4\#V_2 +3\#V_3 +2\#V_4 +\#V_5 &=& 12 + \Sigma_{k\ge 7} (k-6)\#V_k. \label{eq:V2V5}
   \end{eqnarray}    
By hypothesis, $\#V_2 =\#V_3 =\#V_4=0$. The proof is done. 
   \end{proof}
   \begin{remark}
Since $EP \in e\mathcal{MPGN}4$ has all it vertex of degree at least 5, $EP$ has at least 12 vertices are exactly degree 5. What an interesting minimum number is 12! This minimum situation exactly corresponds to an icosahedron which is a convex polyhedron with 20 faces (trianglar facets), 30 edges and 12 vertices. If only $\#V_3$ is non-zero, then $\#V_3=4$ exactly corresponds to a tetrahedron. If only $\#V_4$ is non-zero, then $\#V_4=6$ exactly corresponds to an octahedron. For a general situation without asking some $V_i$ must be $0$, what is a combinatorial description or explanation for Equation~\ref{eq:V2V5}? 
   \end{remark}
   
\subsection{Subgraphs created by an RGB-tiling}

In this subsection, we consider a fixed MPG $M$ with a fixed RGB-tiling $T_{rgb}$.  Let $T_{r}$, $T_{g}$ and $T_{b}$ be the R-tiling, G-tiling and B-tiling induced by $T_{rgb}$ respectively. We also use $T_r$, $T_g$ and $T_b$ to represent the sets of red, green, blue edges respectively. For any vertex $v\in V(M)$, it is nature to define $\deg_r(v)$ to be the numbers of red edges incident to $v$; so to define $\deg_g(v)$ and $\deg_b(v)$.  

First, we focus on $M-T_r$ which is also a bipartite graph made by black edges ${T_r}^{-1}(\text{black})$. The planar graph $M-T_r$ is also made by squares which come from diamonds without red edges in middle. Let us define $\deg^{\bar{r}}(v):= \deg_g(v)+\deg_b(v)$ for every $v\in V(M-T_r)$. The corresponding numbers $\#V^{\bar{r}}$, $\#E^{\bar{r}}$, $\#F^{\bar{r}}$ and $\#V^{\bar{r}}_j$ for $M-T_r$ are nearly duplicated from the previous discussion. Clearly $\#V^{\bar{r}}=\#V$,  $3\#E^{\bar{r}}=2\#E$ and $2\#F^{\bar{r}}=\#F$.
 
The following derivation steps are as same as the previous discussion. We have   
$\#V^{\bar{r}} -\#E^{\bar{r}} +\#F^{\bar{r}}=2$ (Euler's formula), $4\#F^{\bar{r}}=2\#E^{\bar{r}}$ (four black edges form a square), and $\Sigma_{j\ge 2} j \#V^{\bar{r}}_j =2\#E^{\bar{r}}$, and $\#V^{\bar{r}}=\Sigma_{j\ge 2} \#V^{\bar{r}}_j$.   
   \begin{eqnarray}
   4\#V^{\bar{r}}-(2+2)\#E^{\bar{r}}+4\#F^{\bar{r}} &=& 8; \nonumber \\  
   4\Sigma_{j\ge 2} \#V^{\bar{r}}_j -\left(\Sigma_{j\ge 2} j\#V^{\bar{r}}_j + 4\#F^{\bar{r}} \right) +4\#F^{\bar{r}} &=& 8;\nonumber \\ 
   2\#V^{\bar{r}}_2 +\#V^{\bar{r}}_3 &=& 8 + \Sigma_{j\ge 5} (j-4)\#V^{\bar{r}}_j. \label{eq:V2V5square}
   \end{eqnarray}        

The basic requirement of an RGB-tiling is three different edge-colors in every triangle. This nature tells us: 
   \begin{eqnarray}
\delta(\text{$\deg(v)$ is odd}) \le &\deg_r(v)& \le \lfloor \deg(v)/2\rfloor; \nonumber \\ 
\lceil \deg(v)/2 \rceil \le &\deg^{\bar{r}}(v)& \le \deg(v)-\delta(\text{$\deg(v)$ is odd}), \label{eq:V2V5degBar}
   \end{eqnarray}    
where $\delta(\ast)$ is the Kronecker Delta function. Let us denote $\#V^{\bar{r}}_{k,j}$ to be the number of vertices $v$ in $M$ with $\deg(v)=k$ and $\deg^{\bar{r}}(v)=j$. By Inequality~\ref{eq:V2V5degBar}, we have 
   $$
\#V_k=\sum_{j=\lceil k/2 \rceil}^{k-\delta(\text{$k$ is odd})} \#V^{\bar{r}}_{k,j}.
   $$
Especially, we have     
   $$
   \begin{array}{cllllll}
\#V_4 &=& \#V^{\bar{r}}_{4,2} \ +\!\!\!\! & \#V^{\bar{r}}_{4,3} \ +\!\!\!\! & \#V^{\bar{r}}_{4,4};\\
\#V_5 &=&                         & \#V^{\bar{r}}_{5,3} \ +\!\!\!\! & \#V^{\bar{r}}_{5,4};\\ 
\#V_6 &=&                         & \#V^{\bar{r}}_{6,3} \ +\!\!\!\! & \#V^{\bar{r}}_{6,4} \ +\!\!\!\! & \#V^{\bar{r}}_{6,5} \ +\!\!\!\! & \#V^{\bar{r}}_{6,6};\\
\#V_7 &=&                         &                         & \#V^{\bar{r}}_{7,4} \ +\!\!\!\! & \#V^{\bar{r}}_{7,5} \ +\!\!\!\! & \#V^{\bar{r}}_{7,6}.
   \end{array}
   $$ 
We show these identities of $\#V_k$ for the first few $k$. Index $k$ start with $4$, because each MPG that we will deal with only has at most one vertices of degree 4 and the rest of vertices of degree at least 5.

We also use $\#V^R_{k,i}$ to denote the number of vertices in $M$ which are degree $k$ and red-degree $i$; while the superscripts $R$ and $\bar{r}$ for $\#V^R_{k,i}$ and $\#V^{\bar{r}}_{k,j}$ are easy to distinguish. Clearly $\#V^R_{k,i} = \#V^{\bar{r}}_{k,k-i}$.  

Using Equations~\ref{eq:V2V5} and~\ref{eq:V2V5degBar}, we plan to develop a skill to estimate how many degree 6 vertices are neighbors of a degree 5 vertex for any $EP\in e\mathcal{MPGN}$.

\bibliographystyle{amsplain}

\end{document}